\documentclass[12pt,letterpaper,reqno]{amsart}

\usepackage{times}
\usepackage[T1]{fontenc}
\usepackage{mathrsfs}
\usepackage{latexsym}
\usepackage[dvips]{graphics}
\usepackage{epsfig}
\usepackage{amsmath,amsfonts,amsthm,amssymb,amscd}
\input amssym.def
\input amssym.tex

\newcommand\be{\begin{equation}}
\newcommand\ee{\end{equation}}
\newcommand\bea{\begin{eqnarray}}
\newcommand\eea{\end{eqnarray}}
\newcommand\bi{\begin{itemize}}
\newcommand\ei{\end{itemize}}
\newcommand\ben{\begin{enumerate}}
\newcommand\een{\end{enumerate}}
\newcommand\bc{\begin{center}}
\newcommand\ec{\end{center}}
\newcommand\ba{\begin{array}}
\newcommand\ea{\end{array}}

\newtheorem{thm}{Theorem}[section]
\newtheorem{conj}[thm]{Conjecture}

\newtheorem{lem}[thm]{Lemma}

\theoremstyle{definition}
\newtheorem{rek}[thm]{Remark}

\begin{document}

\title{The postage stamp problem and essential subsets in integer bases}

\author{Peter Hegarty}
\email{hegarty@math.chalmers.se} \address{Mathematical Sciences, Chalmers University Of Technology and G\"oteborg University, G\"oteborg, Sweden}


\date{\today}

\begin{abstract} Plagne recently determined the asymptotic behavior of the function
$E(h)$, which counts the maximum possible number of essential elements in an additive basis for $\mathbb{N}$ of order $h$. Here we extend his
investigations by studying asymptotic behavior of the function $E(h,k)$, which counts the maximum possible number of essential subsets of size $k$,
in a basis of order $h$. For a fixed $k$ and with $h$ going to infinity, we show that $E(h,k) = \Theta_{k} \left( [h^{k}/\log h]^{1/(k+1)} \right)$. The
determination of a more precise asymptotic formula is shown to depend on the solution of the well-known $\lq$postage stamp problem' in finite cyclic
groups. On the other hand, with $h$ fixed and $k$ going to infinity, we show that $E(h,k) \sim (h-1) {\log k \over \log \log k}$.
\end{abstract}


\maketitle

\setcounter{equation}{0}


\section{Essential subsets of integer bases}

Let $S$ be a countable abelian semigroup, written additively, $h$ be a positive integer and $A \subseteq S$. The $h$-fold sumset $hA$ consists of all
$s \in S$ which can be expressed as a sum of exactly $h$ not necessarily distinct elements of $A$. If $S$ is infinite, we write $hA \sim S$ if all but
finitely many elements of $S$ lie in $hA$. In that case, $A$ is said to be a {\em basis of order $h$}{\footnote{In the literature, the term {\em
asymptotic basis} is common.} if $hA \sim S$ but $(h-1)A \not\sim S$. If $S$ is finite, then a basis $A$ of order $h$ must satisfy $hA = S$ and $(h-1)A
\neq S$. The two semigroups of interest in this paper (and in most of the additive number theory literature) are $S = \mathbb{N}$, the set of positive
integers, and $S = \mathbb{Z}_{n}$, the set of residue classes modulo a positive integer $n$.
\\
\\
Now suppose $A$ is a basis of some order for $\mathbb{N}$, a so-called {\em integer basis}. A finite subset $E$ of $A$ is said to be an {\em essential
subset} of $A$ if $A \backslash E$ is no longer a basis of any order, and the set $E$ is minimal with this property. In the case when $E$ is a
singleton set, $E = \{a\}$ say, we say that $a$ is an {\em essential element} of $A$.
\\
\\
A fundamental result of Erd\H{o}s and Graham \cite{EG} states that every integer basis possesses only finitely many essential elements. Grekos \cite{G}
refined this observation by showing that the number of essential elements in a basis of order $h$ is bounded by a function of $h$ only. Let $E(h)$
denote the maximum possible number of essential elements in a basis of order $h$. Two recent papers have left us with a very good understanding of this
function. In 2007, Deschamps and Farhi \cite{DF} proved that
\be
E(h) \leq c \sqrt{{h \over \log h}},
\ee
with $c = 30 \sqrt{{\log 1564 \over 1564}}
\approx 2.05$, and gave an example to show that this is the best-possible universal constant. That left the question of asymptotic behavior and, in
2008, Plagne \cite{P} completed the picture by showing that
\be
E(h) \sim 2 \sqrt{{h \over \log h}}.
\ee
Most of his paper was in fact devoted to
verifying that the asymptotic behavior of $E(h)$ is regular.
\\
\\
Deschamps and Farhi appear to be the first people to study essential subsets in an integer basis of arbitrary size. They generalized the
Erd\H{o}s-Graham result by showing that any basis possesses only finitely many essential subsets. However, the number of these cannot be bounded purely
in terms of the order of the basis, as the following example from their paper shows. Let $s \geq 1$ and $p_1,...,p_s$ denote the first $s$ prime numbers.
Put $\mathscr{P} := \prod_{i=1}^{s} p_{i}$ and take
\be A = \mathscr{P} \cdot \mathbb{N} \cup \{1,2,...,\mathscr{P}-1\}
\ee
Clearly $A$ is a basis of order 2, but it possesses $s$
different essential subsets, namely the sets
\be
E_i = \{x \in \{1,...,\mathscr{P}-1\} : (x,p_i) = 1 \}, \;\;\;\; i = 1,...,s.
\ee
Note, however, that as $s$ increases in this example, so do the sizes of the essential subsets $E_i$ (and drastically so !). Deschamps and Farhi suggested that the right
generalisation of (1.1) would be an upper bound for the number of essential subsets of a given size in a basis of a given order. In other words, the function
$E(h,k)$, which  denotes the maximum possible number of essential subsets of size $k$ in an integer basis of order $h$, should be well-defined.
In \cite{He} the
present author proved that this is the case, but made no attempt to obtain precise estimates. Motivated by Plagne's subsequent work, we will in this paper
prove the following two results :

\begin{thm}
For each fixed $h > 0$, as $k \rightarrow \infty$ we have
\be
E(h,k) \sim (h-1) {\log k \over \log \log k}.
\ee
\end{thm}

\begin{thm}
Let the function
$f(h,k)$ be given by
\be
f(h,k) := {k+1 \over k^2} \cdot \sqrt[k+1]{k} \cdot \left( {h^{k} \over \log h} \right)^{{1 \over k+1}}.
\ee
Then, for each fixed $k$, as $h \rightarrow \infty$ we have
\\
(i)
\be E(h,k) \gtrsim f(h,k).
\ee
(ii) There is a number $\underline{R}(k) \in (1/e,1)$, to be defined below, such that
\be
E(h,k) \lesssim \left( {1 \over \underline{R}(k)} \right)^{{k \over k+1}} f(h,k).
\ee
\end{thm}

The problem of estimating the function $E(h,k)$ is intimately connected with the
well-known Postage Stamp Problem (PSP), this being the popular name for the general problem of finding bases which are, in some sense,
the most economical possible.  In Section 2 we present an overview of this problem
and, in particular, define the numbers $\underline{R}(k)$ appearing in (1.8) above.
Note that the exact values of these numbers are not known for any $k > 1$. Theorems 1.1 and 1.2 are proven in Sections 3 and 4 respectively.
All our proofs build on the ideas
in previous papers on this subject and are supplemented by ingredients 
of a mostly technical nature. That of Theorem 1.2 is modeled closely on Plagne's
\cite{P}. The main technical problem he faced was to show that the function $E(h,1)$ behaved regularly, and in his case this was basically due to the
unsatisfactory state of current knowledge concerning the distribution of primes in short intervals. When $k > 1$ that state of affairs continues to
create difficulties, but they will turn out to be less serious than those arising from the gaps in our current understanding of the PSP. These gaps
mean that, not only can we not compute exactly the numbers $\underline{R}(k)$, but we will be unable to prove rigorously what we strongly believe to be true, namely :

\begin{conj}
With notation as in Theorem 1.2, we have in fact that \be E(h,k) \sim \left( {1 \over \underline{R}(k)} \right)^{{k \over k+1}} f(h,k). \ee
\end{conj}

We will summarise these outstanding issues in Section 5.

\setcounter{equation}{0}

\section{The Postage Stamp Problem}

For an up-to-date and much more thorough exposition of the material in this section, including an explanation of the name $\lq$Postage Stamp Problem',
see \cite{HJ2}. A more concise, but older, exposition can be found in \cite{AB}.
\\
\\
Let positive integers $h,k$ be given. The {\em postage stamp number} $n(h,k)$ is the largest integer $n$ such that there exists a $k$-element set $A$
of positive integers satisfying $hA_{0} \supseteq \{0,1,...,n\}$, where $A_0 = A \cup \{0\}$. The problem of determining the numbers $n(h,k)$ is
usually traced back to a 1937 paper of Rohrbach \cite{R}. Historically, two special cases have attracted most attention : either $h$ is fixed and $k
\rightarrow \infty$ or vice versa. The two cases seem to be about equally difficult and the current state of knowledge is about the same in both.
For our applications to essential subsets of bases, it turns out however that we can make do with much less information in the
case when $h$ is fixed. The following estimate, already proven by Rohrbach and valid for any $h$ and $k$, will suffice :
\be
\left( {k \over h} \right)^{h} \leq n(h,k) \leq \left( \begin{array}{c} h + k \\ h \end{array} \right).
\ee
The upper bound in (2.1) is obtained by a simple counting argument, and the lower bound is developed constructively. Regarding the former,
observe that for $h$ fixed and $k$ going to infinity,
\be
\left( \begin{array}{c} h+k \\ h \end{array} \right) = {k^h \over h!} + O(k^{h-1}).
\ee
Now let us turn instead to the case when $k$ is fixed and $h \rightarrow \infty$. St\"{o}hr \cite{S} proved the following analogue of Rohrbach's
estimates :
\be
\left( \lfloor {h \over k} \rfloor + 1 \right)^{k} \leq n(h,k) \leq \left( \begin{array}{c} h+k \\ k \end{array} \right).
\ee
Let{\footnote{We have not seen the numbers defined in equations (2.4), (2.5), (2.10), (2.11), (2.17) and (2.18) introduced explicitly in the existing literature on the PSP.}}
\be
s(h,k) := {1 \over k} \left( {n(h,k) \over h^{k}} \right)^{-1/k}.
\ee
For each fixed $k$, the limit
\be
s(k) := \lim_{h \rightarrow \infty} s(h,k)
\ee
is known to exist \cite{K} and it follows easily from (2.3) that, for each $k$,
\be
{1 \over e} < s(k) \leq 1.
\ee
Only three values are known :
\be
s(1) = 1, \;\;\; s(2) = 1, \;\;\; s(3) = \sqrt[3]{3/4}.
\ee
The first of these is trivial, the second due to St\"{o}hr and the third to Hofmeister \cite{Ho}. For general $k$ the best-known
lower bound on $s(k)$ tends to $1/e$ as $k \rightarrow \infty$, but for upper bounds it follows from work of Mrose \cite{M} that
\be
\limsup_{k \rightarrow \infty} s(k) \leq {1 \over \sqrt[4]{2}}.
\ee
In more recent times, the PSP has received more attention in the setting of
finite cyclic groups, partly because it can then be formulated in terms of diameters of so-called Cayley graphs, which has applications in the theory
of communication networks. We let $N(h,k)$ denote the largest integer $N$ such that there exists a $k$-element subset $A$ of $\mathbb{Z}_{N} \backslash
\{0\}$ satisfying $hA_{0} = \mathbb{Z}_{N}$, where $A_0 = A \cup \{0\}$. It is trivial that
\be
N(h,k) \geq n(h,k) - 1.
\ee
Bounds similar to (2.1) and (2.3) can be easily obtained, so that if we define
\be
S(h,k) := {1 \over k} \left( {N(h,k) \over h^k} \right)^{-1/k},
\ee
\be
\underline{S}(k) := \liminf_{h \rightarrow \infty} S(h,k), \;\;\; \overline{S}(k) := \limsup_{h \rightarrow \infty} S(h,k),
\ee
then it can be shown that
\be
1/e < \underline{S}(k) \leq \overline{S}(k) \leq 1.
\ee
When the limit exists in (2.11), we denote it $S(k)$. Existence
of the limit does not seem to be known in general. Intuitively, the reason why the numbers $N(h,k)$ are more awkward to handle than the $n(h,k)$ is as
follows : If $A$ is a set of integers such that $hA \supseteq \{0,1,...,n\}$, then naturally $hA \supseteq \{0,1,...,m\}$ for any $m < n$ also. But the
corresponding statement for $\mathbb{Z}_{n}$ and $\mathbb{Z}_{m}$ need not be true.
\\
\\
For $k=2$
it is known that the limit exists and that
\be
S(2) = \sqrt{2/3}.
 \ee
The first rigorous proof of this result seems to be in \cite{HJ1}. No other
values of $\underline{S}(k), \overline{S}(k)$ are known. Once again, no general lower bound is known which doesn't tend to $1/e$ as $k \rightarrow
\infty$. The current record for general upper bounds seems to be due to Su \cite{Su} :
\be
\limsup_{k \rightarrow \infty} \overline{S}(k) \leq \sqrt[5]{{17^{5} \over 5^{5} \cdot 7^{4}}}.
\ee
There is a natural $\lq$dual' to the numbers $N(h,k)$. This time, let $N,k$ be given positive integers,
with $N > k$. We define $h(N,k)${\footnote{The notation $d(N,k)$ is common in the literature, since these numbers can be interpreted as {\em diameters}
of Cayley graphs.}} to be the smallest positive integer $h$ such that there exists a basis for $\mathbb{Z}_{N}$ of order $h$ containing $k+1$ elements.
For applications to Cayley graphs and also, as we shall see, to essential subsets of bases, the numbers $h(N,k)$ are a more natural choice to work
with than the
$N(h,k)$. The duality between the two is expressed by the easy relations
\be
t \leq N(h(t,k),k), \;\; {\hbox{and}} \;\; h \geq h(N(h,k),k), \;\;\; {\hbox{for any $t,h,k \in \mathbb{N}$}}.
\ee
A dual to (2.3) proven by Wang and Coppersmith \cite{WC} is the double inequality
\be
\sqrt[k]{k!} \cdot N - {k+1 \over 2} \leq h(N,k) \leq k \cdot (\sqrt[k]{N} - 1).
\ee
The natural counterparts to the numbers $S(h,k), \underline{S}(k), \overline{S}(k)$ are thus
\be
R(h,k) := {h(N,k) \over k \cdot \sqrt[k]{N}},
\ee
\be
\underline{R}(k) := \liminf_{h \rightarrow \infty} R(h,k), \;\;\; \overline{R}(k) := \limsup_{h \rightarrow \infty} R(h,k).
\ee
The numbers $\underline{R}(k)$ are those appearing in Theorem 1.2. From (2.16) we have
\be
1/e < \underline{R}(k) \leq \overline{R}(k) \leq 1.
\ee
Again it is natural to conjecture that the limits always exist and then that $R(k) = S(k)$. All
we can immediately deduce from (2.15), however, is that
\be
\overline{R}(k) \geq \underline{S}(k) \;\; {\hbox{and}} \;\; \underline{R}(k) \leq \overline{S}(k).
\ee
Apart from what can then be deduced from (2.13), (2.14) and (2.20), very little seems to be known, though it was shown in \cite{WC} that
$\underline{R}(2) = S(2) = \sqrt{2/3}$. In particular, existence of the limits $R(k)$ does not seem to be known for a single value of $k > 1$. The
subtle difficulty in handling the numbers $N(h,k)$ referred to above is thus fully reflected in the $h(N,k)$. Tables of values computed in \cite{HJ1}
show that $h(N,k)$ is not even a non-decreasing function of $N$.

\setcounter{equation}{0}

\section{Proof of Theorem 1.1}

Let $A$ be a basis for $\mathbb{N}$ of order $h$ with $s$ essential subsets of size $k$,
say $E_1,...,E_s$. We think of $h$ as being fixed and $k,s$ large. Let $E := \cup_{i} E_{i}$,
$E_{0} := E \cup \{0\}$ and,
for each $i$,
\be
d_i := {\hbox{GCD}} \; \{a - a^{\prime} : a,a^{\prime} \in A \backslash E_{i} \}.
\ee
Then each $d_i > 1$ and these numbers are relatively prime (\cite{DF}, Lemma 12).
So if the $d_i$ are in increasing order, then $d_i \geq p_i$, the $i$:th prime. Let $d := \prod_{i} d_i$. Thus,
\be
d \geq \prod_{i=1}^{s} p_i \gtrsim \left( {s \log s \over \alpha} \right)^{s},
\ee
for some absolute constant $\alpha > 0$. This latter estimate for the product of the first $s$ primes is well-known : see, for example, \cite{Rob}.
\\
\\
Next, let $\alpha_{1},...,\alpha_{s}$ be numbers such that $a \equiv
\alpha_{i} \; ({\hbox{mod $d_i$}})$ for all $a \in A \backslash E_i$. Without loss of generality, each $\alpha_{i} = 0$ (otherwise, choose a negative
integer $\alpha$ such that $\alpha \equiv \alpha_{i} \; ({\hbox{mod $d_i$}})$ for each $i$, and replace $A$ by the shifted set $A - \alpha$). Now
since $A$ is a basis for $\mathbb{N}$ of order $h$, the numbers in $E_{0}$ must, when considered modulo $d$, form a
basis for $\mathbb{Z}_{d}$ of order at most $h-1$. Thus
\be
d \leq N(h-1,ks).
\ee
From (3.2), (3.3), (2.9), (2.1) and (2.2) it is easily verified that
\be
s \lesssim (h-1) {\log k \over \log \log k},
\ee
which proves that the right-hand side of (1.5) is asymptotically an upper bound for $E(h,k)$.
\\
\\
For the lower bound, we turn the above argument on its head. Let $h$ be given and $k$ a very large integer. We wish to
construct a subset $A$ of $\mathbb{N}$ which is a basis of order $h$ and has about $(h-1) {\log k \over
\log \log k}$ essential subsets of size $k$. Our example is modeled on that in \cite{DF}, and presented in Section 1.
To begin with, let $s$ be the largest
integer such that
\be
(hs) \cdot \left( \prod_{i=1}^{s} p_i \right)^{{1 \over h-1}} \leq k.
\ee
From (3.2) we have
\be
s \sim (h-1) \cdot \left( {\log k \over \log \log k} \right).
\ee
Put $\mathscr{P} := \prod_{i=1}^{s} p_i$. By the left-hand inequality in (2.1), there exists a set
\\
$F \subseteq
\{1,...,\mathscr{P}-1\}$ with
\be
|F| \leq (h-1) \cdot \mathscr{P}^{{1 \over h-1}} \leq k
\ee
and such that, considered modulo $\mathscr{P}$, $F_{0}$ is a basis for $\mathbb{Z}_{\mathscr{P}}$ of order $h-1$. For each
$i = 1,...,s$, let $F_{i} := \{x \in F : (x,p_i) = 1\}$. Thus $|F_i| \leq k$ for each $i$ also.
We wish to augment the set $F$ to a set $E$, still
contained inside $\{1,...,\mathscr{P}-1\}$, such that
two conditions are satisfied :
\\
\\
(i) $E_{0}$ is still a basis of order $h-1$ for $\mathbb{Z}_{\mathscr{P}}$, i.e.: it is not a basis of strictly smaller order,
\\
(ii) $|E_{i}| = k$, for $i = 1,...,s$.
\\
\\
Note that, for sufficiently large $k$, (i) will follow from (ii) by the choice of $s$. Let $\mathscr{G} :=
\{1,...,\mathscr{P}-1\} \backslash F$ and, for each $i$,
\be
\mathscr{G}_{i} := \{x \in \mathscr{G} : p_i | x \; {\hbox{and}} \; (x,p_j) = 1 \; {\hbox{for all $j \neq i$}} \}.
\ee
Note that the sets $\mathscr{G}_{i}$ are pairwise disjoint and that, from (3.7) and Mertens theorem,
\be
|\mathscr{G}_{i}| = \Theta \left( {\mathscr{P} \over \log s} \right), \;\;\;\; {\hbox{for $i = 1,...,s$}}.
\ee
Put $f_i := |F_i|$. First of all,
add in at most $s-2$ multiples of $p_{s-1} p_s$ from $\mathscr{G}$ to $F$ so that at this point
\be
{\hbox{$\sum_{i=1}^{s} f_i$ is a multiple of $s-1$}}.
\ee
Now we want to throw in $g_i$ elements of $\mathscr{G}_i$ so that, for each $i$,
\be
f_i + \sum_{j \neq i} g_j = k.
\ee
The unique solution to the linear system (3.11) is
\be
g_i = {k + (s-1)f_i - \sum_{i=1}^{s} f_i \over s-1}
\ee
and, by (3.5), (3.7), (3.9) and (3.10), the right-hand side of (3.12) is a positive integer less than $|\mathscr{G}_i|$
for each $i$, as desired. The set $E$ now consists of $F$ together with all the numbers we have thrown in
during the above process and, by construction, it satisfies (ii). Finally, then, let $A \subseteq
\mathbb{N}$ be given by
\be
A = (\mathscr{P} \cdot \mathbb{N}) \cup E.
\ee
Since $E$ is a basis of order $h-1$ for $\mathbb{Z}_{\mathscr{P}}$, it follows that $A$ is an integer
basis of order $h$. By construction, it has $s$ essential subsets of size $k$, namely the sets
$E_1,...,E_s$. From (3.6) we thus have what we want, and so the proof of Theorem 1.1 is complete.

\setcounter{equation}{0}

\section{Proof of Theorem 1.2}

First we consider the upper bound (1.8). As in the previous section, let $A$ be a basis for $\mathbb{N}$ of order $h$ with $s$ essential subsets of size $k$, say $E_1,...,E_s$. This time we think of $k$ as being fixed and $h$ very large. Let
\be
E_i = \{a_{i,j} : j = 1,...,k\}, \;\;\; i = 1,...,s.
\ee
Let the numbers $d_i$ be as in (3.1), $d := \prod_{i} d_i$ and $A^{*} := A \backslash \left( \cup_{i} E_i \right)$. As before, we can argue that, without loss of generality, $a \equiv 0 \; ({\hbox{mod $d$}})$ for all $a \in A^{*}$. Now, with
the numbers $h(\cdot, \cdot)$ defined as in Section 2, we claim that
\be
h \geq \sum_{i=1}^{s} h(d_i,k).
\ee
To see this, first note that, by
definition of the numbers $h(d_i,k)$, there exist integers $x_i$ such that, for each $i$, no representation
\be
x_i \equiv \sum_{j=1}^{k} \gamma_{i,j} a_{i,j} \; ({\hbox{mod $d_i$}})
\ee
exists satisfying
\be
\gamma_{i,j} \in \mathbb{N}_0, \;\;\; \sum_{j} \gamma_{i,j} < h(d_i,k).
\ee
Now let $x$ be any positive
integer satisfying $x \in hA$ and $x \equiv x_i \; ({\hbox{mod $d_i$}})$ for $i=1,...,s$. Since $x \in hA$ there exists a representation
\be
x = \sum_{i=1}^{s} \sum_{j=1}^{k} \gamma_{i,j} a_{i,j} + \sum_{a \in A^{0}} a,
\ee
where $A^{0}$ is some multisubset of $A^{*}$, each $\gamma_{i,j} \geq 0$
and
\be
h = |A^{0}| + \sum_{i,j} \gamma_{i,j}.
\ee
But reducing (4.5) modulo $d_i$ gives a congruence of the form (4.3) for each $i$. Thus (4.2) follows from (4.4) and (4.6).
\\
\\
Now let $h \rightarrow \infty$. Then
\be
h \geq \sum_{i=1}^{s} h(d_i,k) \gtrsim k \cdot \underline{R}(k) \cdot \sum_{i=1}^{s} \sqrt[k]{d_i}
\ee
and
\be
\sum_{i=1}^{s} \sqrt[k]{d_i} \geq \sum_{i=1}^{s} \sqrt[k]{p_i} \gtrsim \sum_{i=1}^{s} \sqrt[k]{i \log i} \\ \gtrsim \int_{1}^{s} (x \log x)^{1/k} \; dx \gtrsim {k \over k+1} (s^{k+1} \log s)^{1/k},
\ee
where the integral has been easily estimated using partial integration. Summarising, we have shown that
\be
(s^{k+1} \log s)^{1/k} \lesssim \left( {k+1 \over k^2} {1 \over \underline{R}(k)} \right) h.
\ee
Choosing the set $A$ so that $s = E(h,k)$, this is easily checked to yield (1.8).
\\
\\
So to the lower bound (1.7). Once again, we wish to turn the above argument on its head. In \cite{P}, the author considered the case $k=1$. To show that the function $E(h,1)$ behaved regularly, he needed to know that every sufficiently large positive integer could be expressed as
$\sum (p-1)$, the sum being over a particular type of set of prime numbers. In the present context, one should think of $p-1$ as being the number
$h(p,1)$. To generalise the argument directly and prove Conjecture 1.3, it would suffice that, for each $k > 1$, every sufficiently large
integer could be expressed as $\sum h(p,k)$, the sum being over a similar set of primes with the additional property that the numbers $R(p,k)$ approach
$\underline{R}(k)$ as $p \rightarrow \infty$. Of course, if we also knew that the limits $R(k)$ existed, then we wouldn't need to worry about the
latter bit. We do not see how to carry out this procedure, given the current state of knowledge about extremal bases in finite cyclic groups, though we
strongly believe it can be done, perhaps with some small modifications.
Instead, we prove the weaker inequality (1.7) by constructing, for all large primes, a large number of bases for
$\mathbb{Z}_p$ all of which are fairly close to extremal (Theorem 4.4). These bases are sufficiently plentiful to allow us to deal easily with further
technical issues concerning the distribution of primes in short intervals (Theorem 4.3). Now to the details. We begin with a pair of lemmas.
\\
\\
The first is a result of Alon and Frieman also used in \cite{P}. Recall the following notations : If $X$ is a finite subset of $\mathbb{N}$ then
$\Sigma(X)$ denotes the collection of all subset sums from $X$. If $q \in \mathbb{N}$ then we denote $X(q) := \{x \in X : q | x\}$. We also set
\be
S_{X} := \sum_{x \in X} x
\ee
and
\be
B_{X} := \sqrt{\sum_{x \in X} x^2}.
\ee
Then there is the following result :

\begin{lem}\cite{AF}
For each $\epsilon > 0$ there exists $n_0 \in \mathbb{N}$ such that if $n \geq n_0$ and $X \subseteq \{1,...,n\}$ satisfies $|X| > n^{2/3 + \epsilon}$
and $|X \backslash X(q)| \geq n^{2/3}$ for each $q \geq 2$, then
\be
\left\{ \lceil {1 \over 2} S_{X} - {1 \over 2} B_{X} \rceil,...,\lfloor {1 \over 2}
S_{X} + {1 \over 2} B_{X} \rfloor \right\} \subseteq \Sigma(X).
\ee
\end{lem}
Our second lemma will be a rather general result about the representability of sufficiently large integers as a certain type of subset sum in a
sufficiently dense multisubset of $\mathbb{N}$. Here we need to make precise some terminology.
\\
\\
By a {\em multisubset} $A$ of $\mathbb{N}$ we mean a collection of positive integers where repititions are allowed. We assume that each integer occurs
only finitely many times in a multisubset. If $a_1 \leq a_2 \leq …$ are the elements of $A$ written in some non-decreasing order, then we denote this
by $A = (a_i)$. We shall say that $A$ is {\em weakly increasing} if the following holds : for each $\epsilon > 0$ there exists $\delta > 0$ such that,
for all $n >>_{\epsilon} 0$,
\be
{a_{\lfloor (1+\epsilon)n \rfloor} \over a_n} > 1 + \delta.
\ee
If $A$ is a multisubset of $\mathbb{N}$ we denote by
$A^{\#}$ the subset of $\mathbb{N}$ consisting of all those numbers which appear at least once in $A$. Now recall that if $X \subseteq \mathbb{N}$, the {\em lower asymptotic density} of $X$, denoted $\underline{d}(X)$, is defined as
\be
\underline{d}(X) = \liminf_{n \rightarrow \infty} {|X \cap [1,n]| \over n}.
\ee
Our lemma is the following :

\begin{lem}
Let $A = (a_i)$ be a weakly increasing multisubset of $\mathbb{N}$ such that $\underline{d}(A^{\#}) = 1$. Let $\epsilon > 0$. Then for all $h
>>_{\epsilon} 0$, there exists some representation of $h$ as a sum
\be
h = \sum_{i=1}^{n} a_i + \sum_{j \in \mathscr{J}} a_j,
\ee
where $\mathscr{J} \subseteq A^{\#} \cap [a_n,(1+\epsilon)a_n]$. Here $n$ depends on $h$, but $n \rightarrow \infty$ as $h \rightarrow \infty$.
\end{lem}

\begin{proof}
Fix $\epsilon > 0$. For each $n > 0$ set
\be
A^{\#}_{n} := A^{\#} \cap [a_n,a_{\lfloor (1+\epsilon)n \rfloor}].
\ee
Now define the sequence $(u_n)_{n=1}^{\infty}$ by
\be
u_n := \sum_{i=1}^{n} a_i + {1 \over 2} \sum_{j \in A^{\#}_{n}} a_j.
\ee
The sequence $u_n$ is evidently increasing and,
if $n^{\prime} := \lfloor (1+\epsilon)n \rfloor$, then
\be
u_{n+1} - u_n \leq a_{n+1} + {1 \over 2} \left( a_{n^{\prime}+1} + a_{n^{\prime}+2} \right).
\ee
Since $\underline{d}(A^{\#}) = 1$, it follows that
\be
u_{n+1} - u_n \leq (1+O(\epsilon))a_{n}.
\ee
Now let $h$ be a very large integer (how large
$h$ needs to be will become clear in what follows). Let $n$ be the largest integer such that $u_n < h$. Put $h^{\prime} = h - u_n$. By (4.19) we have
that, in the notation of (4.10),
\be
\left| h^{\prime} - {1 \over 2} S_{A^{\#}_{n}} \right| = O(a_n).
\ee
Since $A$ is weakly increasing, when $h$ and
thus $n$ are sufficiently large, there exists $\delta > 0$ such that
\be
{a_{\lfloor (1+\epsilon)n \rfloor} \over a_n} > 1 + \delta.
\ee
Furthermore, since $\underline{d}(A^{\#}) = 1$ then for any $\delta^{\prime} > 0$ and $h >> 0$, the set $A^{\#}_{n}$ will contain at least the fraction $1-\delta^{\prime}$ of all numbers in the interval $[a_n,a_{\lfloor (1+\epsilon)n \rfloor}]$. What all of this means is that $A^{\#}_{n}$ will satisfy the hypotheses of Lemma 4.1 and moreover that, in the notation of (4.11), $B_{A^{\#}_{n}} = \Omega(a_{n}^{3/2})$. Hence, by Lemma 4.1 and (4.20) it follows that, provided $h$ is sufficiently large, there is a subset $\mathscr{J} \subseteq A^{\#}_{n}$ such that $h^{\prime} = \sum_{a_j \in \mathscr{J}} a_j$. From the definition of $h^{\prime}$, this implies (4.15) and so the proof of the lemma is complete.
\end{proof}

Let $\mathbb{P} = (p_i)$ denote the sequence of primes, as usual. We now have :

\begin{thm}
Let $k$ be a positive integer and $\epsilon > 0$. Then for all integers $h >>_{k,\epsilon} 0$, there exists a representation
\be
h = \sum_{i=1}^{n} \lfloor \sqrt[k]{p} \rfloor + \sum_{j \in \mathscr{J}} \lfloor \sqrt[k]{p_j} \rfloor,
\ee
where $\mathscr{J} \subseteq \{n+1,...,\lfloor (1+\epsilon)n \rfloor \}$.
\end{thm}

\begin{proof}
Fix $k$ and $\epsilon$. Let $A$ denote the multisubset of $\mathbb{N}$ consisting of the integer parts of the $k$:th roots of all the primes. To prove
the theorem, we just need to verify that $A$ satisfies the hypotheses of Lemma 4.2. Clearly, $A$ is weakly increasing. It is also the case that
$\underline{d}(A^{\#}) = 1$, in other words, that almost every positive integer is the integer part of the $k$:th root of some prime. While it is
generally believed that, in fact, $A^{\#} = \mathbb{N}$, for any $k \geq 2$, what is known for certain is that $\mathbb{N} \backslash A^{\#}$ is finite
for any $k \geq 3$, and that $\underline{d}(A^{\#}) = 1$ for $k = 2$. These facts are easy consequences of the following two well-known theorems
respectively (in each case the exponents given are the smallest that have been arrived at to date, to the best of our knowledge) :
\\
\\
{\sc Result 1 \cite{H-B}} : As $n \rightarrow \infty$ one has \be \pi(n + t) - \pi(n) \sim {t \over \log n}, \ee whenever $n^{7/12} \leq t \leq n$.
\\
\\
{\sc Result 2 \cite{J}} : For each $\epsilon > 0$ there is a prime in the interval $(n,n+n^{1/20 + \epsilon})$ for almost all positive integers $n$.
\\
\\
Thus our set $A$ does indeed satisfy the hypotheses of Lemma 4.2, and thus the proof of Theorem 4.3 is complete.
\end{proof}

The above will take care of the technicalities arising from the distribution of the primes. We now turn to the construction of reasonably efficient
bases in finite cyclic groups.

\begin{thm}
Let $k \geq 2$ be an integer. There exists an absolute constant $c > 0$, independent of $k$, such that, for all primes $p >>_{k} 0$, and all $s$ such
that $0 \leq s < c \lfloor \sqrt[k]{p} \rfloor$, there exists a set $A$ of $k$ non-zero elements of $\mathbb{Z}_{p}$ such that $A \cup \{0\}$ is a
basis for $\mathbb{Z}_{p}$ of order $k \cdot \lfloor \sqrt[k]{p} \rfloor + s$.
\end{thm}

\begin{rek}
This is overkill for our purposes. It would suffice for us to know that there exist $(k+1)$-element bases for $\mathbb{Z}_p$ of order $k \cdot \lfloor
\sqrt[k]{p} \rfloor + s$ for each $s \in \{0,1\}$. But we think the result as stated may be of independent interest - see Section 5.
\end{rek}

\begin{proof}
Fix $k \geq 2$ and let $p$ be a prime. Let $x := \lfloor \sqrt[k]{p} \rfloor$ and $\epsilon := \sqrt[k]{p} - x$. Thus $\epsilon \in (0,1)$. Our goal is to construct, for some constant $c > 0$ and all each $s \in \{0,1,...,\lfloor cx \rfloor\}$,
a subset $A \subseteq \mathbb{Z}_{p}^{\times}$ of size $k$ such that $A_{0} := A \cup \{0\}$ is a basis
for $\mathbb{Z}_{p}$ of order $kx+s$. By the binomial theorem,
\be
x^{k} = p - \sum_{j=1}^{k} \left( \begin{array}{c} k \\ j \end{array} \right) \epsilon^{j} x^{k-j}.
\ee
In particular, it is clear that, for $p >> 0$ we will have
\be
p - (k+1)x^{k-1} < x < p.
\ee
First consider $A := \{1,x,x^{2},...,x^{k-1}\}$. Then $A_{0}$ is a basis of order $(kx - k) + u$, where $u$ is the smallest integer such that
\be
(x-2) \cdot 1 + (x-1) \cdot x + (x-1) \cdot x^{2} + \cdots + (x-1) \cdot x^{k-2} + (x+u) \cdot x^{k-1} \geq p-1.
\ee
The left-hand side of (4.26) is just $x^{k} + (u+1)x^{k-1} - 2$. Hence, if $p >> 0$, (4.25) implies that $0 \leq u \leq k$. Thus $A_{0}$ is a basis of order $kx-j$ for some $j \in \{0,1,...,k\}$. Now let $t$ be any
integer and consider
\be
A_{t} := \{1,x,x^{2},...,x^{k-2},x^{k-2}(x-t)\}.
\ee
If $t$ is small compared to $x$ then $A_{t,0}$ will be a basis of order
$(kx-k)+(u_{t}-t)$, where $u_{t}$ is the smallest integer such that
\be
(x-2) \cdot 1 + (x-1) \cdot x + (x-1) \cdot x^{2} + \cdots + (x-1) \cdot x^{k-3} + (x-t-1) \cdot x^{k-2} + (x+u_{t}) \cdot x^{k-2}(x-t) \geq p-1.
\ee
Let $v_{t} := u_{t} - t$. We have already seen above that $0 \leq v_{0} \leq k$.
The theorem will be proved if we can show that there are values of $t$ for which $v_{t}$ takes on each of the values $k,k+1,...,k+\lfloor cx \rfloor$, for some absolute constant $c > 0$. After some tedious computation where we make use of (4.24), the inequality (4.28) reduces to
\be
x^{k-2}\left[(u_{t}+1)(x-t) - tx\right] \geq 1 + \sum_{j=1}^{k} \left( \begin{array}{c} k \\ j \end{array} \right) \epsilon^{j} x^{k-j}.
\ee
Note that the right-hand side is independent of $t$. Denote it simply by $\Sigma$ and note from (4.24) that $x^{k} + \Sigma = p+1$. Then from (4.29) we easily deduce that $v_{t} = \lceil f(t) \rceil$, where the real-valued function $f$ of one variable is given by
\be
f(\xi) = {\Sigma + \xi x^{k-1} \over x^{k-2}(x-\xi)} - (\xi+1).
\ee
One easily computes that
\be
f^{\prime}(\xi) = {p+1 \over x^{k-2}(x-\xi)^{2}} - 1,
\ee
hence that
\be
f^{\prime}(\xi) = {1 + o_{p}(1) \over (1 - \xi/x)^{2}} - 1.
\ee
Thus $f$ is increasing in the range $0 \leq \xi < x$, $f^{\prime}(\xi) = \Theta(1)$ when $\xi = \Theta(x)$ and $f^{\prime}(\xi)
\leq 1 + o_{p}(1)$ when $\xi/x \leq 1 - 1/\sqrt{2}$. It follows easily that, as $t$ increases, the integer-valued quantity $v_{t}$ takes on a sequence of $\Theta(x)$ consecutive values, starting at $v_{0}$. This suffices to prove Theorem 4.4.
\end{proof}

Now we are ready to prove inequality (1.7). Let $k \geq 2$ be a fixed integer. Let $h$ be a positive integer and write $h = kh_1 + s$ where $0 \leq s < k$. Let $\epsilon > 0$. If $h >>_{\epsilon,k} 0$ then, by Theorem 4.3 there exists a representation
\be
h_1 = \sum_{i=1}^{n} \lfloor \sqrt[k]{p_{i}} \rfloor + \sum_{j \in \mathscr{J}} \lfloor \sqrt[k]{p_{j}} \rfloor,
\ee
where $\mathscr{J} \subseteq \{n+1,...,\lfloor (1+\epsilon)n \rfloor\}$. For each
prime $p_{i} > p_{k}$ in this sum we wish to choose a $k$-element subset $A_{i}$ of $\{1,2,...,p_i - 1\}$ such that, if we identify $A_{i}$ with a subset of $\mathbb{Z}_{p}$ and let $r_{i}$ denote the order of $A_{i} \cup \{0\}$ as a basis for $\mathbb{Z}_{p}$, then
\be
r_{i} = k \cdot \sqrt[k]{p_{i}} + O(1),
\ee
and
\be \sum r_{i} = h.
\ee
From (4.33) and Theorem 4.4 (see Remark 4.5 in fact), it is clear that such a choice is possible, for sufficiently large $h$. Set
$\mathscr{I} := \{1,...,n\} \cup \mathscr{J}$, $\mathscr{P} := \prod_{i \in \mathscr{I}} p_{i}$ and, for each $i$, $\mathscr{P}_{i} :=
\mathscr{P}/p_{i}$. For each $i \in \mathscr{I} \backslash \{1,...,k\}$ set
\be
A_{i} := \{a_{i,j} : j = 1,...,k\},
\ee
and
\be
E_{i} := \{ a_{i,j} \mathscr{P}_{i} : j = 1,...,k \}.
\ee
Now consider the subset $A \subseteq \mathbb{N}$ given by
\be A = \mathscr{P} \cdot \mathbb{N} \cup \left( \bigcup_{i \in \mathscr{I} \backslash \{1,…,k\}} E_{i} \right).
\ee
By construction, the set $A$ is a basis for $\mathbb{N}$ of order $h$ and contains
$|\mathscr{I}| - k$ essential subsets of size $k$, namely each of the sets $E_{i}$. The proofs of these assertions are similar to those of the
corresponding assertions in \cite{P} (see page 9 of that paper), so we do not include them. For the purpose of obtaining the right-hand side of (1.7)
as a lower bound for the asymptotic behavior of $E(h,k)$, it now suffices to show that
\be
|\mathscr{I}| - k \geq (1 - O(\epsilon)) \left( {k+1 \over k^2} \cdot \sqrt[k+1]{k} \right) \left( {h^{k} \over \log h} \right)^{{1 \over k+1}}.
\ee
First, it is obvious that
\be
|\mathscr{I}| - k = (1 + O(\epsilon))n.
\ee
Second, it follows from (4.33) and (4.34) that
\be
h \leq  (1 + O(\epsilon)) \cdot k \cdot \sum_{i = 1}^{\lfloor (1 + \epsilon)n \rfloor} \sqrt[k]{p_{i}}.
\ee
Hence if we can show that
\be
\sum_{i=1}^{n} \sqrt[k]{p_{i}} \sim {k \over k+1} \left( n^{k+1} \log n \right)^{1/k},
\ee
then this and (4.40)-(4.41) are easily seen to imply (4.39). But (4.42) has already been established in (4.8), and so our proof of Theorem 1.2 is complete.

\setcounter{equation}{0}

\section{Discussion}

We have seen that an entirely satisfactory estimate for the function $E(h,k)$ cannot be obtained without significant progress on the
Postage Stamp Problem in the case when $k$ is fixed and $h \rightarrow \infty$. Specifically, one needs to know the numbers
$\underline{R}(k)$ given by (2.18). Even then, a subtle technicality arises in attempting to rigorously prove Conjecture 1.3, as
was alluded to in Section 4. It is possible, though highly unlikely, that not all sufficiently large integers can be expressed
as sums $\sum h(p,k)$ over certain sets of primes, as in Theorem 4.3. For example, it could happen that $h(p,k)$ was a multiple of $k$ for every $p$. Note that the upper bound in (2.16) has this property, and it was just this fact that necessitated the long detour via Theorem 4.4 when trying to prove (1.7). Theorem 4.4 may be independently interesting in the sense that one can ask a very general question as to what are the possible
orders of an arbitrary $(k+1)$-element basis for $\mathbb{Z}_{n}$. A special case would be to ask for the best-possible $c$ in the statement of that theorem. Does $c \rightarrow \infty$ as $p$ does ?
For the proof of Conjecture 1.3 one would instead like to know what is the largest possible $C = C(p,k)$ such that
there exists a $(k+1)$-element basis for $\mathbb{Z}_{p}$ of order $h(p,k) + s$, for every $0 \leq s \leq C(p,k)$. Can we take
\\
$C(p,k) = \Omega \left( \sqrt[k]{p} \right)$ ?

\section*{Acknowledgement}

I thank Alain Plagne for very helpful discussions and Melvyn Nathanson for some literature tips on the PSP. This work was completed while I
was visiting City University of New York, and I thank them for their hospitality. My research is partly supported by a grant from the
Swedish Research Council (Vetenskapsr\aa det).

\ \\

\end{document}